\documentclass{amsart}
\usepackage{fullpage}
\usepackage{amsmath,amssymb,amsthm}
\usepackage{mathbbol}
\usepackage{graphicx}

\usepackage{amsrefs} 

\usepackage[all,cmtip]{xy}
\theoremstyle{plain}
\newtheorem{thm}{Theorem}[section]
\newtheorem{lem}[thm]{Lemma}
\newtheorem{prop}[thm]{Proposition}
\newtheorem{cor}[thm]{Corollary}
\theoremstyle{definition}
\newtheorem{defn}[thm]{Definition}

\newtheorem{rem}[thm]{Remark}

\newcommand{\lMod}[1]{{{#1}\textrm{-}\mathbf{Mod}}} 
\newcommand{\rMod}[1]{\mathbf{Mod}\textrm{-}{#1}} 
\newcommand{\bimod}[3]{{_{#1}}{#2}_{#3}} 
\newcommand{\arrow}[1]{\stackrel{#1}{\longrightarrow}}
\def\N{\mathbb{N}}
\def\Z{\mathbb{Z}}

\def\k{\mathbb{k}}

\def\TT{\mathsf{T}}

\def\bb{\mathsf{b}}
\def\BB{\mathsf{B}}
\def\NN{\mathsf{N}}
\def\tt{\mathsf{t}}
\def\ss{\mathsf{s}}

\def\Ae{{A^{\mathsf{e}}}} 
\def\Aop{A^{\mathsf{op}}} 
\def\iso{\cong} 
\def\oco{\otimes \cdots \otimes}
\def\lla{\vartriangleright}
\def\lra{\vartriangleleft}
\def\rla{\blacktriangleright}
\def\rra{\blacktriangleleft}
\DeclareMathOperator{\Ext}{Ext}
\DeclareMathOperator{\Tor}{Tor}
\DeclareMathOperator{\id}{id}
\DeclareMathOperator{\im}{im}

\DeclareMathOperator{\Aut}{Aut}
\DeclareMathOperator{\Hom}{Hom}

\begin{document}
\title{Untwisting a Twisted Calabi-Yau Algebra}
\author{Jake Goodman }
\address{University of Glasgow, School of Mathematics and Statistics, University Gardens, Glasgow
G12 8QW, Scotland}
\email{j.goodman.1@research.gla.ac.uk}
\author{Ulrich Kr{\"a}hmer}
\email{ulrich.kraehmer@glasgow.ac.uk}
\begin{abstract} 
Twisted Calabi-Yau algebras are a generalisation of Ginzburg's notion of Calabi-Yau algebras. Such algebras $A$ come equipped with a \textit{modular automorphism} $\sigma \in \Aut(A)$, the case $\sigma = \id$ being precisely the original class of Calabi-Yau algebras. Here 
we prove that every twisted Calabi-Yau algebra may be extended to a Calabi-Yau algebra. More precisely, we show that if $A$ is a twisted Calabi-Yau algebra with modular automorphism $\sigma$, then the smash product algebras $A \rtimes_\sigma \N$ and $A \rtimes_\sigma \Z$ are Calabi-Yau.
\end{abstract}
\maketitle
\tableofcontents
\section{Introduction} 

Twisted Calabi-Yau algebras, that is, algebras $A$ that satisfy a twisted
Poincar\'e duality 
in Hochschild (co)homology, have been the object of intense recent study
due to their natural prevalence in various flavours of noncommutative geometry,
see e.g. \cites{BergerSolotar,BSW,BrownZhang,Dolgushev,Kr1,LR,LiuWu,LiuWangWu,ReyesRogalskiZhang}. More precisely, such algebras have isomorphic Hochschild homology and cohomology up to a twist in the coefficients
given by some algebra automorphism $ \sigma \in \mathrm{Aut}(A)$,   
\[H^\bullet(A,M) \iso H_{d-\bullet}(A,\bimod{\sigma}{M}{})\]
This generalises Ginzburg's notion of a Calabi-Yau algebra
\cite{Ginzburg} which cover the case where $\sigma = \id$.
See the main text for notation and precise definitions.

By analogy with the situation for Poisson manifolds \cite{Dolgushev} or 
in Tomita-Takesaki theory \cite{connes} 
it is natural to expect that
the smash product $ A \rtimes_\sigma \mathbb{Z} $ is an (untwisted)
Calabi-Yau algebra. Some partial results in that direction have been
obtained in, or can be derived from \cites{Farinati,LiuWangWu,ReyesRogalskiZhang}. However, to the best
of our knowledge no general proof of this fact has been
obtained yet. Specifically, we make no assumptions about the existence
of a grading on $A$, or about properties of the 
automorphism $ \sigma $, e.g. finiteness of order or more generally, semisimplicity.
Hence the aim of the present paper is to prove:  
\begin{thm}\label{maintheorem}
Let $A$ be a $\sigma$-twisted Calabi-Yau algebra of dimension $d$. Then $A \rtimes_\sigma \N$ and $A \rtimes_\sigma \Z$ are Calabi-Yau algebras of dimension $d+1$.
\end{thm}

The structure of the paper is as follows: In the first section, we provide some background on the Hochschild (co)homology of algebras and the cup and cap products. In Section 2, 
we recall a theorem of Van den Bergh on Poincar\'e type duality in Hochschild (co)homology and define twisted Calabi-Yau algebras. The third section describes twisted cyclic homology and the underlying notion of a paracyclic module of which the Hochschild complex $C_\bullet(A,\bimod{\sigma}{A}{})$ of an algebra $A$ with coefficients twisted by an algebra automorphism $\sigma$ is an example. We also show, as a consequence of the basic paracyclic 
theory, that the Hochschild homology $H_\bullet(A,\bimod{\sigma}{A}{})$ is invariant under the natural action of $\sigma$ for any algebra $A$ and $\sigma \in \Aut(A)$ which implies the invariance of the cohomology $H^\bullet(A,M)$ of a twisted Calabi-Yau algebra under the action of the modular automorphism. In Section 4, 
we introduce the smash products $A \rtimes_\sigma \N$ and $A \rtimes_\sigma \Z$ and recall a result of Farinati \cite{Farinati} which shows that the Calabi-Yau property of $A \rtimes_\sigma \Z$ is implied by that of $A \rtimes_\sigma \N$. The section is then concluded by the 
proof of the main theorem.

The second author acknowledges 
support by the EPSRC grant ``Hopf Algebroids
and Operads'' and the Newton Institute Cambridge which
provided a great environment for finishing this paper.

\section{Preliminaries}
\subsection{General Notations}
Fix a field $\k$ of characteristic zero. All algebras under consideration will be unital and associative $\k$-algebras. Unadorned $\otimes$ and $\Hom$ shall denote $\otimes_\k$ and $\Hom_\k$ respectively. If $A$ is an algebra, we shall denote by $\Aop$ and $\Ae:=A \otimes \Aop$ its opposite and enveloping algebras. The antipodal map 
\[S\colon \Ae \longrightarrow \Ae\] 
satisfying $S(a\otimes b)=(b \otimes a)$ is an involution and induces inverse functors \[S\colon\lMod{\Ae} \rightleftarrows \rMod{\Ae}\colon \! S\]
identifying the categories of left and right $\Ae$-modules. There are further equivalences between the categories of left and right $\Ae$-modules and the category of $A$-bimodules with a symmetric action of $\k$. If $M$ and $N$ are left (resp. right) $\Ae$-modules, we shall denote by $\lla$ and $\lra$ (resp. $\rla$ and $\rra$) the induced left and right actions of $A$. 
\subsection{Twisted Bimodules}
Let $M$ be an $A$-bimodule. For each pair of automorphisms $\rho,\sigma \in \Aut(A)$, we define the \textit{twisted bimodule} $\bimod{\rho}{M}{\sigma}$ to be the $\k$-space $M$ together with the $A$-bimodule structure:
\[a \cdot m \cdot b:= \rho(a)m\sigma(b),\quad a,b \in A,m \in M\]
If either $\rho$ or $\sigma$ is the identity, we shall supress it from the notation, writing for example $\bimod{}{M}{\sigma}$ instead of $\bimod{\id}{M}{\sigma}$. The following lemma is standard.
\begin{lem}
Let $\rho,\sigma,\tau \in \Aut(A)$ be automorphisms. Then
\begin{itemize}
\item[(i)] The map 
\[\bimod{\rho}{A}{\sigma} \arrow{} \bimod{\tau\circ\rho}{A}{\tau\circ\sigma}, \quad a \mapsto \tau(a)\] 
is an isomorphism of bimodules. In particular, $\bimod{}{A}{\sigma}\iso\bimod{\sigma^{-1}}{A}{}$.
\item[(ii)] The bimodules $A$ and $\bimod{}{A}{\sigma}$  are isomorphic if and only if $\sigma$ is an inner automorphism.
\end{itemize}
\end{lem}
\begin{proof}
See, for example, \cites{BrownZhang}.
\end{proof}
\subsection{Hochschild Homology}
The \textit{Hochschild homology} of an algebra $A$ with coefficients in a right $\Ae$-module $M$ is given by \[H_\bullet(A,M):=\Tor^{\Ae}_\bullet(M,A)\] and is realised as the homology of the \textit{Hochschild (chain) complex} $(C_\bullet(A,M),\mathsf{b}_\bullet)$ where \[C_n(A,M):= M \otimes A^{\otimes n}\] and the boundary map $\mathsf{b}_\bullet$ satisfies
\begin{align*}
\mathsf{b}_n(m \otimes a_1 \oco a_n) &= m \rra a_1 \otimes  a_2 \oco a_n \\
                                     & \quad + \sum_{i=1}^{n-1} (-1)^i m \otimes a_1 \oco a_ia_{i+1} \oco a_{n} \\
                                     & \quad + (-1)^n a_n \rla m \otimes a_1 \oco a_{n-1}.
\end{align*}
\subsection{Hochschild Cohomology}
The \textit{Hochschild cohomology} of an algebra $A$ with coefficients in a left $\Ae$-module $N$ is
\[H^\bullet(A,N):=\Ext_{\Ae}^\bullet(A,N)\]
and is realised as the cohomology of the \textit{Hochschild (cochain) complex} $(C^\bullet(A,N),\mathsf{b}^\bullet)$ where \[C^n(A,N):= \Hom_\k(A^{\otimes n}, N)\]
and the coboundary map 
$\mathsf{b^\bullet}$ satisfies
\begin{align*}
\mathsf{b}^n(f)(a_1 \otimes \cdots \otimes a_{n+1})  &= a_1 \lla f(a_2 \oco a_{n+1}) \\
                                            &+ \sum_{i=1}^n (-1)^i  f(a_1 \oco a_ia_{i+1} \otimes \cdots a_{n+1}) \\
                                            &+ (-1)^{n+1} f(a_1 \oco a_n) \lra a_{n+1}.
\end{align*}
\subsection{Cup Product}
The Hochschild cohomology $H^\bullet(A,A)$, \textit{a priori} just a $\k$-space, may also be considered as a graded commutative algebra whose multiplication 
\[\cup\colon H^\bullet(A,A)\otimes H^\bullet(A,A) \arrow{} H^\bullet(A,A),\]
called the \textit{cup product}, is induced by the associative product 
\[C^n(A,A)\otimes C^m(A,A)\arrow{} C^{n+m}(A,A),\qquad f \otimes g \mapsto f \cup g\]
 on the Hochschild complex $C^\bullet(A,A)$ where $f \cup g\colon A^{\otimes (n+m)}\arrow{} A$ is the map satisfying: 
\[(f\cup g)(a_1 \oco a_n \otimes a_{n+1} \oco a_{n+m}):=(-1)^{mn}f(a_1
\oco a_n)g(a_{n+1}\oco a_{n+m}).\]
\subsection{Cap Product}
The \textit{cap product} is the map
\[\cap\colon H_p(A,M)\otimes H^n(A,N) \arrow{} H_{p-n}(A,M\otimes_A N),\quad p \geq n\]
defined on the level of Hochschild (co)chains by
\[m \otimes a_1 \oco a_p \cap f= (-1)^{pn}m\otimes_A f(a_1 \oco a_n)\otimes a_{n+1} \oco a_{p}.\]
The tensor product $M \otimes_A N$ is formed using the right and left
actions $\bimod{}{M}{\rra}$ and $\bimod{\lla}{N}{}$ respectively and
is considered as a right $\Ae$-module using the remaining actions
$\bimod{\rla}{M}{}$ and $\bimod{}{N}{\lra}$.
\begin{rem}
Taking the cohomology coefficients in the above to be $N=A$, the cap product endows the Hochschild homology $H_\bullet(A,M)$ with the structure of a (graded) right $H^\bullet(A,A)$-module via
\[H_p(A,M)\otimes H^n(A,A) \arrow{\cap} H_{p-n}(A,M\otimes_A A) \arrow{\iso} H_{p-n}(A,M).\]
\end{rem}
\section{Duality and the Fundamental Class}
\subsection{Van den Bergh's Theorem}
The most general framework for Poincar\'e type duality in Hochschild (co)homology is provided by a well known theorem of Van den Bergh which we state below. First, we need the following definition.
\begin{defn}
An algebra $A$ is said to be \textit{homologically smooth} if there exists a finite length resolution of the $\Ae$-module $A$ by finitely generated projective $\Ae$-modules. 
\end{defn}
\begin{thm}(\cites{VdB})
Let $A$ be a homologically smooth algebra. Assume that there exists an integer $d \geq 0$ such that $\mathrm{Ext}^i_\Ae(A,\Ae) \iso 0$ for all $i \neq d$ and further assume that $U_A:= H^d(A,\Ae)$ is an invertible $\Ae$-module. Then, for all left $\Ae$-modules $M$, there are natural isomorphisms
\[H^\bullet(A,M)\iso H_{d-\bullet}(A,U_A \otimes_A M).\]
\end{thm}
\begin{rem}
When the conditions of the theorem are satisfied, the algebra $A$ is said to have \textit{Van den Bergh duality of dimension $d$} and the (right) $\Ae$-module $U_A:= H^d(A,\Ae)$ is called the \textit{dualising bimodule} of $A$. In this case, we necessarily have that $d$ is equal to the dimension $\dim(A)$ of $A$ which is, by definition, the projective dimension of $A$ as an $\Ae$-module (see \cites{CartanEilenberg}).
\end{rem}
\subsection{The Fundamental Class}
For an $n$-dimensional closed orientable manifold $M$, the Poincar\'e duality isomorphism \[H^\bullet(M)\iso H_{n-\bullet}(M)\] in singular cohomology is given by taking the cap product with a fundamental class $[M] \in H_n(M)$. In \cites{Lambre} and \cites{KoKr2} it is shown that the dualising isomorphism in the (co)homology of an algebra with Van den Bergh duality may be realised analogously.

First, we must define the fundamental class. Given any algebra $A$ and any 
integer $d \geq 0$, abbreviate $U_A:=H^d(A,\Ae)$ as in the statement of Theorem~{2.2} (however, for the moment we do not assume any of the further conditions to be satisfied). Then, the cap product
\[H_d(A,U_A)\otimes U_A \arrow{} H_0(A,U_A \otimes_A \Ae)\iso U_A\]
provides a $\k$-linear map
\[F\colon H_d(A,U_A) \arrow{} \Hom_\Ae(U_A,U_A), \quad F(z) = (z \cap -).\]
In \cite{Lambre}, it is proven that if $A$ has Van den Bergh duality of dimension $d$, then $F$ is an isomorphism. One then defines the \textit{fundamental class} of $A$ to be the unique element $\omega_A \in H_d(A,U_A)$ such that $F(\omega_A)=\id$. Using this terminology, we then have:
\begin{prop}
If $A$ is an algebra with Van den Bergh duality of dimension $d$ and $M$ is any left $\Ae$-module, then
\[\omega_A \cap -: H^\bullet(A,M) \arrow{} H_{d-\bullet}(A,U_A \otimes_A M)\]
is an isomorphism.
\end{prop}
\begin{proof}
See \cite{Lambre} or \cite{KoKr2}.
\end{proof}
\begin{rem}
Taking $M=A$, the proposition says that $H_\bullet(A,U_A)$ is freely generated as a $H^\bullet(A,A)$-module by the fundamental class $\omega_A$.
\end{rem}
\subsection{Twisted Calabi-Yau Algebras}
In \cites{Ginzburg}, Ginzburg introduced the study of Calabi-Yau algebras, a noncommutative generalisation of the coordinate rings of Calabi-Yau varieties. Specifically, a Calabi-Yau algebra is an algebra $A$ with Van den Bergh duality whose dualising bimodule $U_A$ is isomorphic to $A$ as right $\Ae$-modules. More generally, we have:
\begin{defn}
An algebra $A$ is said to be \textit{$\sigma$-twisted Calabi-Yau} if it has Van den Bergh duality and the dualising bimodule $U_A$ is isomorphic to $\bimod{}{A}{\sigma}$ for some $\sigma \in \Aut(A)$. The automorphism $\sigma$ is variously called the \textit{modular} or \textit{Nakayama automorphism}\footnote{The term \textit{modular automorphism} was used in \cites{Dolgushev} since in the case of a deformation quantization of a Poisson variety, the modular automorphism quantizes the flow of the modular vector field of the Poisson structure whereas the authors in \cites{BrownZhang} use the term \textit{Nakayama automorphism} since for a Frobenius algebra, it coincides with the classical Nakayama automorphism.} of $A$. 
\end{defn}
\begin{rem}
Observe that the modular automorphism is only unique up to inner automorphisms. If $\sigma$ is the identity (or more generally an inner automorphism), then the algebra $A$ is a Calabi-Yau algebra in the sense of Ginzburg. 
\end{rem}
\begin{rem}If $A$ is a $\sigma$-twisted Calabi-Yau algebra, any choice of isomorphism $U_A \arrow{} \bimod{\sigma^{-1}}{A}{}$ identifies the fundamental class $\omega_A \in H_d(A,U_A)$ with an element of $H_d(A,\bimod{\sigma^{-1}}{A}{})$ which we shall refer to as \textit{a} fundamental class for $A$ and by abuse of notation, shall denote by $\omega_A$. 
\end{rem}
The twisted Calabi-Yau condition appears to have first been explicitly defined in \cites{BrownZhang} where the authors used the term \textit{rigid Gorenstein}. They showed that a wide class of noetherian Hopf algebras, including for example, the quantised function algebras $\mathcal{O}_q(G)$ of connected complex semisimple algebraic groups $G$, are what we now refer to as twisted Calabi-Yau. Other examples of twisted Calabi-Yau algebras include quantum homogeneous spaces such as Podle\'{s} quantum 2-sphere \cites{Kr1}, the deformation quantization algebra of a Calabi-Yau 
Poisson variety \cites{Dolgushev}, Koszul algebras whose Koszul dual is Frobenius \cites{VdB}, group algebras of Poincar\'e duality groups (e.g. the group algebras of fundamental groups of closed aspherical manifolds) \cites{Lambre} and AS-regular algebras \cites{YekZhang}.

\section{Twisted Cyclic Homology and the Paracyclic Structure of $C_\bullet(A,\bimod{\sigma}{A}{})$} 
\subsection{(Twisted) Cyclic Homology}
The cyclic cohomology of a noncommutative algebra was introduced in the early 1980s independently by Connes \cite{connes1} and Tsygan \cite{Tsygan}. Shortly afterwards, the corresponding theory of cyclic homology was formulated by Loday and Quillen \cite{LodayQuillen}. The cyclic homology of an algebra $A$ is based on the \textit{cyclic operator} \[\tt_\bullet\colon C_\bullet(A,A)\arrow{} C_\bullet(A,A)\] on the Hochschild complex whose $n$th degree component is given by the signed permutation 
\[a_0 \otimes a_1 \oco a_n \mapsto (-1)^n a_n \otimes a_0 \oco a_{n-1}.\]
This endows the simplicial module $C_\bullet(A,A)$ 
with the additional structure of a \textit{cyclic module}. Associated to this structure is \textit{Connes-Tsygan boundary map} $\BB\colon C_\bullet(A,A) \arrow{} C_{\bullet+1}(A,A)$ which anticommutes with the Hochschild boundary $\bb_\bullet$ and leads to Connes' mixed $(\bb,\BB)$-bicomplex whose total homology is the cyclic homology $HC_\bullet(A)$ of $A$. We refer the reader to \cites{Loday} for more detail and precise definitions.

In \cite{KMT}, in order to develop an appropriate analogue of the trace of a $C^\ast$-algebra for the compact quantum groups
of Woronowicz, the authors defined the twisted Hochschild and cyclic cohomologies of an algebra $A$ relative to an automorphism $\sigma \in \Aut(A)$. These extend the usual Hochschild and cyclic cohomology of $A$, which may then be viewed as the corresponding twisted theories relative to the automorphism $\sigma = \id$. 

The dual theories of twisted Hochschild and cyclic homology were then observed to be often less degenerate than the untwisted theories and in particular to avoid the so called `dimension drop' in Hochschild homology that occurs in examples of quantum deformations. 
More precisely, for $q=1$ the only automorphism for which the twisted Hochschild homology of $\mathcal{O}_q(SL_2)$ does not vanish in degree 3 is by the Hochschild-Kostant-Rosenberg theorem the identity. However,
for $q$ not a root of unity this happens precisely for the positive powers of the modular automorphism of $\mathcal{O}_q(SL_2)$, see \cite{HadKr}. This phenomenon has also been observed in similar examples, see e.g.~\cite{Sit}.

\subsection{Paracyclic $\k$-modules}
We now define the notion of a paracyclic module \cites{GetzlerJones} (introduced independently under the name `duplicial module' in \cite{DwyKa}) 
which generalises 
cyclic modules and underlies twisted cyclic homology. First, we recall the following:

\begin{defn}
A \textit{simplicial $\k$-module} is a graded $\k$-module $C_\bullet = \bigoplus_{n\in\N}C_n$ together with maps 
\[\mathsf{d}_{n,i}\colon  C_n \arrow{} C_{n-1}, \quad \mathsf{s}_{n,j}\colon  C_n \arrow{} C_{n+1}\]
for all $n \in \N$ and $0 \leq i,j \leq n$ such that
\[\mathsf{d}_{n-1,i}\mathsf{d}_{n,j}=\mathsf{d}_{n-1,j-1}\mathsf{d}_{n,i}, \quad i<j, \qquad \mathsf{s}_{n-1,i}\mathsf{s}_{n,j} = \mathsf{s}_{n-1,j+1}\mathsf{s}_{n,i}, \quad i \leq j,\]
\[\mathsf{d}_{n+1,i}\mathsf{s}_{n,j} = \begin{cases}\mathsf{s}_{n-1,j-1}\mathsf{d}_{n,i}, & i<j \\ \id, & i=j,j+1 \\ \mathsf{s}_{n-1,j}\mathsf{d}_{n,i-1}, & i >j+1. \end{cases}\]
\end{defn}
\begin{defn}
A \textit{paracyclic $\k$-module} is a simplicial $\k$-module $C_\bullet$ together with maps $\tt_n\colon C_n \arrow{} C_n$ for all $n\in \N$ satisfying
\[\mathsf{d}_{n,i}\mathsf{t}_n = \begin{cases}-\mathsf{t}_{n-1}\mathsf{d}_{n,i-1}, & 1 \leq i \leq n \\  (-1)^n \mathsf{d}_{n,n}, & i=0 \end{cases}, \qquad \mathsf{s}_{n,i}\mathsf{t}_n = \begin{cases}-\mathsf{t}_{n+1}\mathsf{s}_{n,i-1}, & 1 \leq i \leq n \\  (-1)^n\mathsf{t}^2_{n+1}\mathsf{s}_{n,n}, & i = 0.\end{cases}\]
A \textit{cyclic $\k$-module} is a paracyclic $\k$-module $C_\bullet$ such that $\mathsf{T}_n:=\mathsf{t}_n^{n+1} = \id$. In particular, for every paracyclic $\k$-module $C_\bullet$, we define the \textit{associated cyclic module} $C^\mathrm{cyc}_{\bullet} := C_\bullet/\im(\id - \TT_\bullet)$.
\end{defn} 
The motivating example of a paracyclic module is the Hochschild complex $C_\bullet(A,\bimod{\sigma}{A}{})$ associated to an algebra $A$ together with an automorphism $\sigma \in \Aut(A)$. Indeed, one easily checks the paracyclic relations to prove:
\begin{prop}
The Hochschild complex $C_\bullet(A,\bimod{\sigma}{A}{})$ together with the maps $(\mathsf{d}_{n,i},\mathsf{s}_{n,j},\mathsf{t}_n)$ satisfying
\[\mathsf{d}_{n,i}(a_0 \oco a_n) = \begin{cases} a_0 \oco a_ia_{i+1} \oco a_n, & 0\leq i < n \\ \sigma(a_n)a_0 \otimes a_1 \oco a_{n-1}, & i=n \end{cases}\]
\[\mathsf{s}_{n,i}(a_0 \oco a_n) = a_0 \oco a_i \otimes 1 \otimes a_{i+1} \oco a_n\]
\[\mathsf{t}_{n}(a_0 \oco a_n) = (-1)^n\sigma (a_n) \otimes a_0 \oco a_{n-1}\]
is a paracyclic module. \qed
\end{prop}
\subsection{The Homotopy Formula and Quasicyclic Modules}
For any paracyclic module $C_\bullet$ we define the following maps:
\begin{itemize}
\item[(i)] The \textit{simplicial} and \textit{acyclic boundaries}
\[\bb_n := \sum_{i=0}^{n} (-1)^i \mathsf{d}_{n,i} \quad \mbox{and} \quad \bb'_n := \sum_{i=0}^{n-1} (-1)^i \mathsf{d}_{n,i};\]
\item[(ii)] The \textit{norm operator}
\[\NN_n := \sum_{i=0}^n \mathsf{t}^i_n;\]
\item[(iii)] The \textit{extra degeneracy}
\[\ss_{n} := (-1)^{n+1} \mathsf{t}_{n+1}\mathsf{s}_{n,n};\]
\item[(iv)] The \textit{Connes-Tsygan boundary map} 
\[\mathsf{B}_n := (\id - \mathsf{t}_{n+1})\ss_n\mathsf{N}_n. \]
\end{itemize}
Unlike as is the case for cyclic modules, for a paracyclic module 
$\BB_\bullet$ does 
in general neither anticommute with the simplicial boundary $\bb_\bullet$, nor does it square to zero. However, we do have the following:
\begin{prop}[\cite{GetzlerJones}*{Theorem 2.3 (2)}]
For any paracyclic module $C_\bullet$, the simplicial boundary $\bb_\bullet$ and the Connes-Tsygan boundary map $\BB_\bullet$ satisfy the relations 
\[\bb_{n+1}\BB_n+\BB_{n-1}\bb_n = \id - \TT_n,\quad
\BB_{n+1}\BB_n=(\id - \TT_{n+2})(\id - \mathsf{t}_{n+2})\mathsf{s}_{n+1}\mathsf{s}_n \mathsf{N}_n.\]
\end{prop}
\begin{proof}
One establishes the formula directly using the subsiduary relations
\[\bb_n(\id-\tt_n)  = (\id-\tt_{n-1})\bb'_n, \qquad \bb'_{n+1}\ss_{n} + \ss_{n-1}\bb'_n= \id, \qquad \bb'_n\NN_n = \NN_{n-1}\bb_n\]
which are easily verified by computation. 
\end{proof}
The proposition shows that for any paracyclic module $C_\bullet$, the images of $\bb_\bullet$ and $\BB_\bullet$ on the associated cyclic complex $C^\mathrm{cyc}_\bullet$ anticommute. The cyclic homology of $C_\bullet$ is then defined to be the total homology of Connes' mixed $(\bb,\BB)$-bicomplex 
\[\xymatrix{ \ar[d]_{\bb_\bullet}&\ar[d]_{\bb_\bullet} &\ar[d]_{\bb_\bullet} &\ar[d]_{\bb_\bullet} \\
C^\mathrm{cyc}_3 \ar[d]_{\bb_\bullet}&\ar[l]_{\BB_\bullet} C^\mathrm{cyc}_2 \ar[d]_{\bb_\bullet}& \ar[l]_{\BB_\bullet}C^\mathrm{cyc}_1 \ar[d]_{\bb_\bullet}&\ar[l]_{\BB_\bullet} C^\mathrm{cyc}_0 \\
C^\mathrm{cyc}_2 \ar[d]_{\bb_\bullet}&\ar[l]_{\BB_\bullet} C^\mathrm{cyc}_1 \ar[d]_{\bb_\bullet}& \ar[l]_{\BB_\bullet} C^\mathrm{cyc}_0 &                  \\
C^\mathrm{cyc}_1 \ar[d]_{\bb_\bullet}& \ar[l]_{\BB_\bullet} C^\mathrm{cyc}_0 &                  &                   \\
C^\mathrm{cyc}_0 &                  &                  &                   }\]
and the Hochschild homology is defined to be the homologies of the
columns. In case $C_\bullet = C_\bullet(A,\bimod{\sigma}{A}{})$, 
this is denoted by $HH^\sigma_\bullet(A)$. In general, the Hochschild homology differs from the simplicial homology of $C_\bullet$ (which is for $C_\bullet = C_\bullet(A,\bimod{\sigma}{A}{})$ the Hochschild homology $H_\bullet(A,\bimod{\sigma}{A}{})$) but they are isomorphic if $C_\bullet$ is \textit{quasicyclic} (see Definition 3.6).
One can also interpret Proposition 3.4 as the fact that $\BB_\bullet$ is a chain homotopy  $\TT_\bullet \simeq \id$ for the associated complex computing the simplicial homology underlying the paracyclic structure. 

Returning to the example of the Hochschild 
complex $C_\bullet(A,\bimod{\sigma}{A}{})$, the above proposition implies the following invariance property of the fundamental class of a twisted Calabi-Yau algebra which will be key to the proof of Theorem~\ref{maintheorem}:
\begin{cor}\label{fixed} 
If $A$ is a $\sigma$-twisted Calabi-Yau algebra of dimension $d$, then the fundamental class $\omega_A \in H_d(A,{}_{\sigma^{-1}} A)$ is invariant under the action of $\TT_d$, that is, if $z \in C_d(A,{}_{\sigma^{-1}}A)$ 
is any cycle representing $\omega_A$, then we have 
\[[\TT_d(z)]=\TT_d([z])=\TT_d(\omega_A)=\omega_A.\]
\end{cor}
In general, it might not be possible to choose a representative of the fundamental class that is invariant under $\TT_\bullet$ on the chain level. However, if $C_\bullet(A,\bimod{\sigma^{-1}}{A}{})$ is a \textit{quasi-cyclic module} in accordance with the following definition, then it is possible to choose an invariant representative for $\omega_A$.
\begin{defn}[\cites{KoKr1}]
A paracyclic $\k$-module $C_\bullet$ is said to be \textit{quasi-cyclic} if
\[C_\bullet= \ker(\id - \TT_\bullet)\oplus \im(\id-\TT_\bullet).\]
\end{defn}
Indeed, for a quasicyclic module $C_\bullet$, it follows that the projection 
\[(C_\bullet,\bb_\bullet) \arrow{} (C_\bullet^\mathrm{cyc},\bb_\bullet)\]
of simplicial complexes is a quasi-isomorphism, so the homologies of the two complexes are isomorphic to each other and by construction also to the homology of the $\TT_\bullet$-invariants (see \cite{HadKr}*{Proposition 2.1}). 

The motivating example is the paracyclic module $C_\bullet(A,\bimod{\sigma}{A}{})$ associated to any algebra $A$ and automorphism $\sigma$ that acts semisimply on $A$. Indeed, if $\sigma$ is semisimple then 
so is $\TT_\bullet$ and the required splitting comes from the eigenspace decomposition of $C_\bullet(A,\bimod{\sigma}{A}{})$.
\section{The Smash Products $A \rtimes_\sigma \N$ and $A \rtimes_\sigma \Z$}
\subsection{Smash Product Algebras}
Let $A$ be an algebra and let $\sigma \in \Aut(A)$. Then $\sigma$ generates an action of the group algebra $\k[\Z]$ given by $x^i\cdot a:= \sigma^i(a)$ where we identify $\k[\Z]$ with the Laurent polynomials $\k[x^{\pm 1}]$. The \textit{smash product} or \textit{skew group algebra} $A \rtimes_\sigma \Z$ is the $\k$-vector space $A \otimes \k[x^{\pm 1}]$ with multiplication given by the rule
\[(a \otimes x^i)\cdot(b \otimes x^j):= a\sigma^i(b) \otimes x^{i+j}, \qquad a,b \in A \mbox{ and } i,j \in \Z.\]
The smash product or \textit{skew semigroup} algebra $A \rtimes_\sigma \N$ is the subalgebra of $A \rtimes_\sigma \Z$ consisting of all sums of pure tensors of the form $a \otimes x^k$ where $k \in \N$. 
\begin{rem}
For all smash products $A \rtimes_\sigma \N$, there is an isomorphism
\[A \rtimes_\sigma \N \arrow{\iso} A[x;\sigma], \quad a \otimes x^i \mapsto ax^i\]
where the \textit{skew polynomial algebra} $A[x;\sigma]$ is formed by adjoining the variable $x$ to $A$ subject to the commutation relation $xa = \sigma(a)x$. Similarly, the smash products $A \rtimes_\sigma \Z$ are isomorphic to the \textit{skew Laurent algebras} $A[x^{\pm 1};\sigma]$. Whenever it causes no confusion, we shall tacitly make use of these identifications writing, for example $x$ to denote $1 \otimes x \in A \rtimes_\sigma \N$.
\end{rem}
\begin{rem}
The smash products 
$A \rtimes_\sigma \Z$ are examples of the general notion of the smash product $A \rtimes H$ of a Hopf algebra $H$ with an $H$-module algebra $A$. We refer to \cite{Montgomery}*{Chapter 4.} for a full definition. This construction still makes sense if $H$ is only a bialgebra which justifies our use of terminology in calling $A \rtimes_\sigma \N$ a smash product.
\end{rem}
 
\subsection{Noncommutative Localisation}
The following is a short excursus, recalling a general notion of noncommutative localisation proposed by Farinati in \cites{Farinati}, of which the algebra extension $A \rtimes_\sigma \N \hookrightarrow A \rtimes_\sigma \Z$ is an example. It is also proven in \cites{Farinati} that Van den Bergh duality is stable under such localisation and furthermore the dualising module of a localisation is explicitly described. In particular, it follows that the Calabi-Yau property of $A \rtimes_\sigma \Z$ is implied by that of $A\rtimes_\sigma \N$.
\begin{defn}
A map of $\k$-algebras $A \arrow{} B$ is a \textit{localisation} (in the sense of Farinati) if
\begin{itemize}
\item[(i)] The map $A \arrow{} B$ is flat; that is $B \otimes_A -$ and $- \otimes_A B$ are exact functors
\item[(ii)] Multiplication in $B$ induces an isomorphism $B \otimes_A B \arrow{} B$ of $B^\mathsf{e}$-modules.
\end{itemize}
\end{defn}
\begin{thm}
Let $A$ be an algebra with Van den Bergh duality of dimension $d$ and dualising bimodule $U_A$ and let $A \arrow{} B$ be a localisation such that $B \otimes_A U_A \iso U_A \otimes_A B$ as $\Ae$-modules. Then, $B$ has Van den Bergh duality of dimension $d$ with dualising bimodule $B \otimes_A U_A \otimes_A B$. \qed
\end{thm}
\begin{cor}
If $A \rtimes_\sigma \N$ is Calabi-Yau, then $A \rtimes_\sigma \Z$ is Calabi-Yau.
\end{cor}
\begin{proof}
That $A \rtimes_\sigma \N \arrow{} A\rtimes_\sigma \Z$ is a localisation is \cite{Farinati}*{Example 8.}. The conclusion then follows from the description of dualising bimodule provided by the theorem.
\end{proof}
\subsection{The Calabi-Yau property of $A \rtimes_\sigma \N$.}
In this section we shall prove that the smash product $A\rtimes_\sigma \N$ of a $\sigma$-twisted Calabi-Yau algebra $A$ is Calabi-Yau using the following result:

\begin{thm}[\cite{LiuWangWu}*{Propositions 3.1. and 3.2.}]\label{wigan}
Let $A$ be a $\sigma$-twisted Calabi-Yau algebra of dimension $d$. Then the smash product $A \rtimes_\sigma \N$ is homologically smooth and
\[H^{\bullet}(A \rtimes_\sigma \N,(A \rtimes_\sigma \N)^\mathsf{e}) \iso \begin{cases} H^d(A,A\otimes\bimod{}{A}{\sigma^{-1}})\otimes \k[x], & \bullet = d+1 \\ 0, & \bullet \neq d+1\end{cases}\] 
as $(A \rtimes_\sigma \N)^\mathsf{e}$-modules where the actions of $A \rtimes_\sigma \N$ on $H^d(A,A\otimes\bimod{}{A}{\sigma^{-1}})\otimes \k[x]$ are given by
\begin{align*}
a \rla ([f] \otimes x^k) &= \sigma^{-k}(a)\cdot [f] \otimes x^k \\
x \rla ([f] \otimes x^k) &= [f] \otimes x^{k+1} \\
([f] \otimes x^k) \rra a &= [f] \cdot a \otimes x^k \\
([f] \otimes x^k) \rra x &= [(\sigma^{-1})^{\otimes 2} \circ f \circ \sigma^{\otimes d}] \otimes x^{k+1}.
\end{align*}
Here the actions of $A$ on $H^d(A,A\otimes\bimod{}{A}{\sigma^{-1}})$ are induced by those of the right $\Ae$-module structure on the coefficients $A\otimes\bimod{}{A}{\sigma^{-1}}$. 
\qed
\end{thm}
\begin{rem}
It follows from this 
theorem that $A \rtimes_\sigma \N$ is twisted Calabi-Yau and that the modular automorphism extends the identity automorphism of $A$. Such automorphisms of $A \rtimes_\sigma \N$ are parameterised by central units $u$ of $A$ where the automorphism associated to $u$ satisfies $x \mapsto ux$ (cf. \cite{Farinati}*{Proposition. 22.}). It therefore remains to show that the modular automorphism is the one corresponding to $u=1$.
\end{rem}

\begin{proof}[Proof of Theorem~\ref{maintheorem}]
Fix an isomorphism $U_A \arrow{} \bimod{\sigma^{-1}}{A}{}$. Recall that the choice of such an isomorphism fixes a fundamental class $\omega_A \in H_d(A,\bimod{\sigma^{-1}}{A}{})$. The cap product with $\omega_A$ thus yields a $\k$-linear isomorphism
\[\xymatrix{H^d(A,A\otimes\bimod{}{A}{\sigma^{-1}})\otimes \k[x] \ar[rr]^-{(\omega_A \cap -) \otimes \id} & & H_0(A,\bimod{\sigma^{-1}}{A}{} \otimes \bimod{}{A}{\sigma^{-1}}) \otimes \k[x] \ar[rr]^-{([a \otimes b]\mapsto ba)\otimes \id}& & A \otimes \k[x].}\]
The $(A\rtimes_\sigma \N)^\mathsf{e}$-module structure on
$H^d(A,A\otimes\bimod{}{A}{\sigma^{-1}})\otimes \k[x]$ from 
Theorem~\ref{wigan}  
then induces a right action of $(A\rtimes_\sigma
\N)^\mathsf{e}$ on $A \otimes \k[x]$ by transport of structure: it
is clear that for all $(a \otimes x^k) \in A \otimes \k[x]$ and $b
\in A$, this action is given by
\begin{align*}
b \rla (a \otimes x^k) &= \sigma^{-k}(b)a \otimes x^k \\
x \rla (a \otimes x^k) &= a \otimes x^{k+1} \\
(a \otimes x^k) \rra b &= ab \otimes x^k
\end{align*}
together with the right action of $x$ on $A \otimes \k[x]$ that we shall now determine.

Pick a cycle $z \in C_d(A,\bimod{\sigma^{-1}}{A}{})$ such that $\omega_A=[z]$. Then, we have the commutative diagram of (co)chain complexes:
\[\xymatrix{C^\bullet(A,A\otimes A_{\sigma^{-1}})\otimes \k[x] \ar[d]_-{(\sigma^{-1})^{\otimes 2}\circ - \circ \sigma^{\otimes d}\otimes \id}\ar[rrr]^-{(z \cap -) \otimes \id}& & & C_{d-\bullet}(A,\bimod{\sigma^{-1}}{A}{} \otimes A_{\sigma^{-1}}) \otimes \k[x] \ar[d]^-{(\sigma^{-1})^{\otimes d-\bullet+2}\otimes\id} & \\
C^\bullet(A,A\otimes A_{\sigma^{-1}})\otimes \k[x] \ar[rrr]^-{(\TT_d(z) \cap -) \otimes \id}& & & C_{d-\bullet}(A,\bimod{\sigma^{-1}}{A}{} \otimes A_{\sigma^{-1}})\otimes \k[x]}\]
where $\k[x]$ is viewed as a (co)chain complex concentrated in degree
$0$ and $\TT_d(z) = (\sigma^{-1})^{\otimes d}(z)$ comes from the
paracyclic structure on $C_\bullet(A,\bimod{\sigma^{-1}}{A}{})$
described in Proposition 4.3. Now Corollary~\ref{fixed} gives  
\[[\TT_d(z)]=\omega_A,\] 
so the induced diagram on (co)homology takes in degree $d$ the form 
\[\xymatrix{ H^d(A,A\otimes A_{\sigma^{-1}})\otimes \k[x] \ar[d]_-{\rra x} \ar[rr]^-{(\omega_A \cap -)\otimes \id} & &  H_0(A,\bimod{\sigma^{-1}}{A}{} \otimes A_{\sigma^{-1}})\otimes \k[x]\ar[rr]^-{([a \otimes b]\mapsto ba)\otimes \id} & & A \otimes \k[x] \ar[d]^-{\sigma^{-1}\otimes \id}\\
\ar[rr]^-{(\omega_A \cap -)\otimes \id} H^d(A,A\otimes A_{\sigma^{-1}}) \otimes \k[x] &  &  H_0(A,\bimod{\sigma^{-1}}{A}{} \otimes A_{\sigma^{-1}}) \otimes \k[x] \ar[rr]^-{([a \otimes b]\mapsto ba)\otimes \id} & &A \otimes \k[x]}.\]
Consequently, the right action of $x$ on $A \otimes \k[x]$ is given by
\[(a \otimes x^k) \rra x = \sigma^{-1}(a) \otimes x^{k+1}.\]
Finally, the bijection 
\[A \otimes \k[x] \arrow{} A \rtimes_\sigma \N, \qquad a\otimes x^k \mapsto \sigma^{k}(a) \otimes x^k\]
yields an isomorphism $H^{d+1}(A \rtimes_\sigma \N,(A \rtimes_\sigma \N)^\mathsf{e})\iso A \rtimes_\sigma \N$ as right $(A \rtimes_\sigma \N)^\mathsf{e}$-modules thus completing the proof.
\end{proof}
\begin{rem}
In \cites{Farinati}, using a spectral sequence due to Stefan \cites{Stefan}, it is shown that if $H$ is a Calabi-Yau Hopf algebra and $A$ is an $H$-module algebra with Van den Bergh duality, then the smash product $A \rtimes H$ also has Van den Bergh duality although the argument used there does not 
completely determine the dualising bimodule $U_{A \rtimes H}$. Considering the special case where $H=\k[\Z]$ and $A$ is a $\sigma$-twisted Calabi-Yau algebra whose $H$-module structure is given by the modular automorphism $\sigma$, a straightforward adaptation of the main argument in the proof above shows 
that the terms $H^\bullet(A,(A\rtimes_\sigma \Z)^\mathsf{e})$ of the spectral sequence are invariant under $\sigma$, which can be used to fully determine the dualising module and show that it is indeed isomorphic to $A \rtimes_\sigma \Z$. The same result does not directly apply to smash products with the Calabi-Yau bialgebra $\k[\N]$ since in particular it does not 
satisfy the hypotheses of Stefan's spectral sequence. However, one can still use the arguments of \cites{Stefan} to show that a similar spectral sequence for smash products $A \rtimes_\sigma \N$ exists and then the proof in \cites{Farinati} proceeds in the same way. 

However, referring to Theorem~\ref{wigan} allowed us to formulate the proof in a shorter and more 
direct way. The results in \cite{LiuWangWu} show more generally that any Ore extension $A[x;\sigma, \delta]$ of a twisted Calabi-Yau algebra $A$ is also twisted Calabi-Yau. The authors prove this result by means of a bicomplex computing $A[x;\sigma, \delta]$, based on one from \cite{Guccione}. Similarly as with \cite{Farinati}, the proof does not explicitly determine the modular automorphism. In the case where $\delta=0$ so that the Ore extension $A[x;\sigma, \delta]$ is the skew polynomial ring $A[x;\sigma]\iso A \rtimes_\sigma \N$, this bicomplex is an explicit construction of the spectral sequence for $A \rtimes_\sigma \N$ along the lines of Stefan's one as mentioned above. Specialising further to the case where $\sigma$ is the modular automorphism of $A$, one obtains Theorem~\ref{wigan} as stated here.
\end{rem}
\begin{rem}
One interesting feature of the actions of $A \rtimes_\sigma \N$ on $A \otimes \k[x]$ is that they would appear to be more naturally stated as a left rather than right $(A \rtimes_\sigma \N)^\mathsf{e}$-module structure. Although this distinction may seem unimportant, it becomes more pronounced when one views Hochschild (co)homology as an instance of the (co)homology of Hopf algebroids as in \cite{KoKr2}. The duality isomorphism takes right cohomology modules and produces left modules in homology. We refer the interested reader to \cites{KoKr1,KoKr2} and the references contained therein for further details.
\end{rem}

\end{document}